\documentclass[a4paper,oneside]{amsart}
\usepackage{graphicx}

\usepackage{amsmath}
\usepackage{xcolor}
\usepackage{amsthm}
\usepackage{graphicx}
\usepackage{amsfonts}
\usepackage{bbm}

\newtheorem{definition}{Definition}[section]
\newtheorem{theorem}{Theorem}[section]
\newtheorem{corollary}[theorem]{Corollary}
\newtheorem*{remark}{Remark}
\newtheorem{lemma}[theorem]{Lemma}
\usepackage{mathtools}
\usepackage{mathabx}
\usepackage{geometry}

\newcommand{\N}{\mathbb{N}}
\newcommand{\R}{\mathbb{R}}

\newcommand{\lnn}{\ell^2(\mathbb{N})}

\DeclareMathOperator*{\argmin}{arg\,min}
\DeclareMathOperator\supp{supp}
\DeclareMathOperator\diag{diag}
\title{Infinite dimensional generative sensing}
\author{Paolo Angella} 
\address{MaLGa Center, DIMA, University of Genoa, Via Dodecaneso  35, 16146, Genova, Italy}
\email{paolo.angella@edu.unige.it}

\author{Vito Paolo Pastore} 
\address{MaLGa Center, DIBRIS, University of Genoa, Via Dodecaneso  35, 16146, Genova, Italy}
\email{vito.paolo.pastore@unige.it}
\author{Matteo Santacesaria}
\address{MaLGa Center, DIMA, University of Genoa, Via Dodecaneso  35, 16146, Genova, Italy}
\email{matteo.santacesaria@unige.it}
\usepackage[colorlinks=true, allcolors=blue]{hyperref}
\begin{document}

\begin{abstract}
Deep generative models have become a standard for modeling priors for inverse problems, going beyond classical sparsity-based methods. However, existing theoretical guarantees are mostly confined to finite-dimensional vector spaces, creating a gap when the physical signals are modeled as functions in Hilbert spaces. This work presents a rigorous framework for generative compressed sensing in Hilbert spaces. We extend the notion of local coherence in an infinite-dimensional setting, to derive optimal, resolution-independent sampling distributions. Thanks to a generalization of the Restricted Isometry Property, we show that stable recovery holds when the number of measurements is proportional to the prior's intrinsic dimension (up to logarithmic factors), independent of the ambient dimension. Finally, numerical experiments on the Darcy flow equation validate our theoretical findings and demonstrate that in severely undersampled regimes, employing lower-resolution generators acts as an implicit regularizer, improving reconstruction stability.
\end{abstract}

\maketitle


\section{Introduction}
The recovery of high dimensional signals from limited measurements is a fundamental challenge in signal processing, medical imaging, and scientific computing. Classical compressed sensing theory \cite{candes2006robust, donoho2006compressed} revolutionized this field by establishing that signals admitting a sparse representation in some basis can be recovered from a number of linear measurements proportional to their sparsity, rather than their ambient dimension, up to a logarithmic factor.

With the advent of machine learning, the rigidity of the sparsity assumption has been challenged by the emergence of deep generative models \cite{kingma2013auto, goodfellow2014generative, sohl2015deep}.
A substantial body of work in the literature has demonstrated that generative priors where the unknown signal is restricted to lie in the range of a low dimensional generator $G\colon \mathbb{C}^k \to \mathbb{C}^N$ can outperform sparsity-based classical methods, often requiring significantly fewer measurements \cite{pmlr-v70-bora17a,9943805,Berk22, berk2023model,adcock2023unified}. This aligns with the manifold hypothesis \cite{fefferman2016testing}: signals concentrate on low-dimensional geometries that exhibit deformation stability and anisotropic regularity \cite{mallat2012group, grohs2014parabolic}.

In the finite dimensional setting, theoretical guarantees for these approaches rely on Restricted Isometry Properties \cite{candes2005decoding, baraniuk2008simple}, which ensure that the measurement operator preserves in some way the geometry of the generator's range.
The problem lies in the fact that many inverse problems of practical interest such as those arising in Partial Differential Equations, Magnetic Resonance Imaging, or Computed Tomography are inherently infinite dimensional \cite{engl1996regularization}.
The underlying signals are functions in Hilbert spaces, not fixed-length vectors as the literature usually assumes.

Discretizing these problems \textit{a priori} leads to distinct theoretical and practical issues, such as the ``inverse crime'' or a dependency on mesh resolution that obscures the true complexity of the signal \cite{alberti2022infinite}.
This makes the generalization of the deep prior approach to compressed sensing dependent on a proper infinite-dimensional formulation. That is exactly the objective of this work.

We model the inverse problem as recovering an unknown function $x_0 \in \ell^2(\mathbb{N})$ from noisy, weighted evaluations of a unitary operator $\mathcal{F}$ (e.g., the Fourier transform).
We assume $x_0$ lies in the range of a Generalized Generative Network $G\colon \mathbb{C}^k \to \ell^2(\mathbb{N})$, a class of deep ReLU networks acting between finite dimensional spaces with a final linear layer that has $\ell^2(\N)$ as co-domain.
Taking the finite-dimensional results of \cite{berk2023model} as our starting point, this work generalizes their framework to the infinite-dimensional setting.

In this paper we give the following contributions:
\begin{itemize}
    \item We introduce an infinite dimensional notion of local coherence, which quantifies the alignment between the generative model's range and the measurement basis. We derive an optimal, coherence based sampling distribution $p^*$ that minimizes the sample complexity.
    \item We establish a Gen-RIP for infinite dimensional operators. We prove that the number of samples $m$ required for a stable embedding of the generator's range scales logarithmically with the intrinsic dimension $k$ and polynomially with the local coherence $\mu_\mathcal{U}$, independent of the ambient dimension.
    \item We extend our analysis to consider probability distributions that are not supported on the entire domain. We introduce a balancing property to control the approximation error in these practical scenarios.
    \item  Through numerical experiments on Darcy flow, we demonstrate that coherence based sampling significantly outperforms uniform sampling. Furthermore, we uncover a counter intuitive phenomenon where lower resolution generators yield superior reconstruction accuracy in undersampled regimes, suggesting that limited resolution acts as a beneficial implicit regularizer.
\end{itemize}

Finally, we remark that while our framework is formulated in an infinite-dimensional ambient space, the generative priors considered here are inherently finite-dimensional, mapping a low-dimensional latent code to a finite-dimensional subspace of $\ell^2(\mathbb{N})$.

\subsection{Related work}

\textit{Infinite dimensional compressed sensing.}
The transition from finite to infinite dimensional compressed sensing was driven by the need to avoid the so called inverse crime and the discretization errors found in classical formulations. The theory of generalized sampling established that stable recovery of functions is possible using genuinely infinite dimensional measurement models, provided that discretization is handled properly \cite{adcock2012generalized} (see \cite{alberti2022infinite} for extensions to nonlinear inverse problems). Based on this, a theory of compressed sensing in infinite dimension was developed \cite{adcock2016generalized,adcock2017breaking,adcock2018infinite,adcock2022sparse,adcock_efficient_2024}. This framework has since been extended to a variety of bases and sampling strategies. In particular, rigorous recovery guarantees have been derived for sparse expansions in Legendre polynomial bases \cite{rauhut2012sparse}, as well as for sampling schemes exhibiting local or structured behavior, for which oracle type local recovery guarantees have been established \cite{adcock2020oracle}. More recently, a unified theory of infinite dimensional compressed sensing for inverse problems has been developed \cite{alberti2023compressed,alberti2025compressed}. This framework enables the recovery of signals from local measurements (such as those arising from the Radon transform) without discretizing the forward operator, although the associated guarantees primarily focus on wavelet-based sparsity rather than learned generative priors.

\textit{Infinite dimensional generative models.}
The extension of deep generative models to function spaces has emerged as an active area of research, motivated by the need for resolution invariant representations. While early approaches largely focused on scaling discrete architectures, more recent work formulates generative modeling directly in Hilbert spaces. Representative examples include infinite dimensional Variational Autoencoders and score-based diffusion models defined on function spaces, which aim to learn probability measures over infinite dimensional domains \cite{franzese2025generative, bond-taylor2024inftydiff}. Particularly relevant to our setting is the class of Continuous Generative Neural Networks \cite{alberti2022continuous}, and of Functional Autoencoders \cite{bunker2025autoencoders}, which are explicitly designed to learn resolution independent mappings between finite dimensional latent spaces and function spaces (e.g., $L^2(\mathbb{R}^d)$). This makes them a natural candidate for the generator~$G$ in our theoretical framework.

\textit{Implicit Generative Priors.} In the last few years, score-based generative models and diffusion probabilistic models have emerged as a dominant paradigm for solving inverse problems \cite{song2021scorebased,chung2023diffusion}. Unlike the explicit low-dimensional generator $G: \mathbb{C}^k \to \mathcal{H}$ employed in our framework, these methods typically utilize an implicit prior given by score function, enabling high-quality reconstruction through stochastic sampling processes such as Langevin dynamics (see \cite{daras2024survey} for a recent review). While these approaches often achieve state-of-the-art perceptual quality, our work focuses on the rigorous recovery guarantees enabled by the deterministic geometry of explicit low-dimensional manifolds.

\textit{Operator Learning and Scientific Machine Learning.}
Our work is also relevant to the rapidly growing field of scientific machine learning, and in particular to operator learning \cite{boulle2024mathematical}, where neural networks are trained to approximate mappings between infinite dimensional function spaces. Architectures such as the Fourier Neural Operator (FNO) \cite{darcy-dataset} and DeepONet \cite{lu2021learning} have demonstrated strong empirical performance in the solution of parametric partial differential equations. See \cite{lanthaler2022error,kovachki2023neural} for error analysis and approximation theory results. These methods share our goal of resolution independence, but they typically rely on large supervised datasets to learn the forward or inverse operator directly. In contrast, our theory offers a complementary perspective, with the potential to inform optimal data acquisition strategies for training neural operators or for recovering inputs in data-scarce regimes.

\subsection{Structure of the paper}
The rest of the paper is organized as follows. Section~\ref{settupp} formulates the infinite-dimensional inverse problem and defines the class of Generalized Generative Networks used as priors. We introduce the concept of infinite-dimensional local coherence to derive optimal sampling strategies and establish the balancing property to handle non-admissible distributions, and we present our main recovery guarantees. Section~\ref{sec:numerics} validates this framework numerically by applying Functional Autoencoders \cite{bunker2025autoencoders} to the Darcy flow problem. Section~\ref{sec:proofs} contains the proofs of the main results, including the derivation of the Generalized Restricted Isometry Property (Gen-RIP) and the error bounds for the recovered signal. Finally, Section~\ref{sec:conclusions} summarizes the contributions and discusses limitations and open directions.

\section{Setup and Main Results}\label{settupp}

In this section, we establish the mathematical framework for solving inverse problems in infinite dimensional spaces using generative priors. Suppose we have a physical process modeled by an infinite dimensional unitary operator $F \colon \lnn \longmapsto \lnn$. Our objective is to recover an unknown signal $x_0 \in \lnn$ given a noisy measurement vector $b$, such that:
\begin{equation*}
    b=SDFx_0 + \eta,
\end{equation*}
where $S$ is a subsampling operator, $D$ a weighting operator (both defined below) and $\eta$ represents measurement noise.

Classical Compressed Sensing relies on sparsity in a fixed basis. Here, we depart from sparsity and instead assume that the signal $x_0$ lies near the range of a deep generative model. Specifically, we constrain the solution to reside in the range of a generator:
\begin{equation*}
    G \colon \mathbb{C}^k \longmapsto \lnn.
\end{equation*}
To formalize this prior, we introduce the definition of a generalized generative network suitable for our setting.

\begin{definition}[$(k,d)$-Generalized Generative Network]
    Fix the integers $2\leq k \coloneqq k_0 \leq ... \leq k_d$. We fix weight matrices $W^{(i)} \in \mathbb{C}^{k_i \times k_{i-1}}$, as well as a linear map $W \colon \mathbb{C}^{k_d} \longmapsto \lnn$. We define a $(k,d)$-Generalized Generative Network as a function $G\colon\mathbb{C}^k \longmapsto \lnn$ given by:
    \begin{equation*}
        G(z) \coloneqq W \sigma \left( W^{(d)} \sigma\left( \cdot\cdot\cdot W^{(2)} \sigma \left( W^{(1)} z \right)\right) \right),
    \end{equation*}
    where $\sigma$ is the ReLU activation function, defined as:
    \begin{equation*}
        \sigma(x) = \max(0, x).
    \end{equation*}
\end{definition}
This is the natural extension of classical generative networks whose codomain is finite dimensional, the only difference is that we apply, as the final layer, a linear map $W$ into an infinite dimensional space. Crucially, we restrict the intermediate linear maps $W^{(i)}$ to finite-dimensional codomains. Allowing them to map to infinite-dimensional spaces would render the generator's range infinite-dimensional, thereby invalidating the theoretical argument presented in the sequel.
In an infinite dimensional setting, uniform sampling is not achievable. Therefore, we must adopt a variable density sampling strategy. We introduce a probability distribution $p$ on $\mathbb{N}$ which tells us how indices are selected. To analyze this process, we define the associated sampling, weighting, and projection operators.
\begin{definition}
Let $p$ be a probability distribution on $\mathbb{N}$. Let $\{e_j\}_{j \in \mathbb{N}}$ denote the canonical basis of $\ell^2(\mathbb{N})$.
We define the sampling operator $S: \ell^2(\mathbb{N}) \to \mathbb{C}^m$ as the linear map composed of $m$ i.i.d. random vectors $s_1, \dots, s_m$ in $\ell^2(\mathbb{N})$, such that
\begin{equation*}
\mathbb{P}(s_i = e_j) = p_j, \quad \forall j \in \mathbb{N}, \forall i \in [m],
\end{equation*}
where we have denoted $[m]=\{1,\dots,m\}$. The action of $S$ is defined component-wise by the inner product:
\begin{equation*}
(Sx)_i = \langle x, s_i \rangle, \quad \forall i \in [m].
\end{equation*}
\end{definition}


To ensure that our estimator remains unbiased and to normalize the energy of the sampled components, we utilize a diagonal weighting operator $D$.

\begin{definition}
    Given $p$ a probability distribution of $\mathbb{N}$ we define a weight operator $D:\lnn \longmapsto \lnn$ so that
    \begin{equation*}
        (Dx)_i = \begin{cases} \frac{x_i}{\sqrt{p_i}}, & \text{for } p_i>0 \\ 0, & \text{for } p_i=0 \\ \end{cases}
    \end{equation*}
\end{definition}

Furthermore, we identify the support of our sampling distribution via the operator $\hat{I}$. This operator essentially acts as a projection onto the indices that have a non-zero probability of being sampled.

\begin{definition}
    Given $p$ a probability distribution of $\mathbb{N}$ we define the operator $\hat{I}:\lnn \longmapsto \lnn$
    \begin{equation*}
        \hat{I} \coloneqq \diag(\mathbbm{1}_{\supp p}).
    \end{equation*}
\end{definition}

With these operators in place, our recovery strategy is formulated as an optimization problem. We seek a latent vector $\hat{z}$ that minimizes the data consistency error in the measurement domain:
\begin{equation*}
    \hat{z} \in \argmin_{z \in {\mathbb{C}} ^{k}} \| b - SDF G(z) \|.
\end{equation*}
We outline that $D$ here is simply a weight operator needed to balance $p$ since in general it will lead to a nonuniform sampling. The recovered signal is then given by:
\begin{equation*}
    \hat{x} = G(\hat{z}).
\end{equation*}
Our ultimate analytical goal is to derive an upper bound on the reconstruction error $\| \hat{x} - x_0 \|$.
The overall setup is similar to the finite dimensional case, with the key difference that we can handle both the infinite dimensional nature of the data and probability distributions that are not supported on the entire domain.
In particular, the ability to deal with distributions that fail to sample certain relevant coordinates is also important in the finite dimensional setting, but it has not been addressed in previous works in the literature.

\subsection{Local Coherence and Sampling Strategies}

To guarantee a better recovery rate, the sampling distribution $p$ must be matched to the ``geometry" of the operator $F$ relative to the signal model. If $F$ maps the signal model to specific coordinates in $\lnn$, we must sample those coordinates with higher probability. This alignment is captured by the concept of \textit{local coherence}, borrowed from classical compressed sensing \cite{candes2007sparsity} and extended to generative models in \cite{Berk22}.

Let $\mathcal{U}\subseteq\lnn$ be a union of $N$ convex cones $\bigcup_{i = 1}^N \mathcal{C}_i$. We define $\Delta(\mathcal{U}) \coloneqq \bigcup_{i = 1}^N \text{span}(\mathcal
{C}_i)$. As for intuition, we should think about $\mathcal{U}$ as the range of our generator
\begin{definition}\label{local-coherence}
    Let $\{e_i\}_{i\in\mathbb{N}}$ denote the canonical basis of $\ell^2(\mathbb{N})$. Given a unitary operator $F$ on $\ell^2(\mathbb{N})$, we define its $i$-th row as the vector $F_i := F^* e_i$.
    Given a finite union of convex cones $\mathcal{U}\subseteq\ell^2(\mathbb{N})$ contained in a finite-dimensional subspace, we define the local coherence $(\alpha_i)_{i\in\mathbb{N}}$ of $\mathcal{U}$ with respect to $F$ as
    \begin{equation} \label{def:loccoh}
        \alpha_i = \sup_{x\in \Delta(\mathcal{U}) ,\|x\|=1} |\langle x, F_i\rangle| = \sup_{x\in \Delta(\mathcal{U}) ,\|x\|=1} |(Fx)_i|.
    \end{equation}
\end{definition}

We can show that this is well defined in our framework of finite dimensional subspaces with the following lemma.
\begin{lemma}\label{alpha-in-l2}
    Let $F \colon \lnn \longmapsto \lnn$ be a unitary operator and let $\mathcal{U}\subseteq\ell^2(\mathbb{N})$ be a finite union of convex cones contained in a subspace of dimension $r$, $\alpha=(\alpha_i)_{i\in\mathbb{N}} $ the local coherence of $\mathcal{U}$ as defined in \eqref{def:loccoh}. Then
    \begin{equation*}
        \alpha \in \lnn \quad \text{and} \quad\| \alpha \|^2 \leq r^2.
    \end{equation*}
\end{lemma}
    This is trivial in the finite case but not in the infinite dimensional one.
The local coherence $\alpha_i$ quantifies the maximum correlation between the $i$-th measurement basis vector and any normalized difference vector in the model set. A sampling distribution is considered ``good'' if it sufficiently covers indices with high coherence.

\begin{definition}
    Given $F$ unitary in $\lnn$ and a finite union of convex cones $\mathcal{U}\subseteq\ell^2(\mathbb{N})$ contained in a finite-dimensional subspace, and $p$ a probability distribution on $\mathbb{N}$. Let $(\alpha_i)_{i\in\mathbb{N}}$ be the local coherence of $F$ with respect to $\mathcal{U}$. We will call $\mu_{\mathcal{U}}(F,p)$ the smallest constant for which
    \begin{equation*}\label{bound}
        \alpha_j \leq \mu_{\mathcal{U}}(F,p) \sqrt{p_j} \quad \forall j \colon p_j > 0
    \end{equation*}
    holds. If $\mu_{\mathcal{U}}(F,p)$ is such that
    \begin{equation*}
        \alpha_j \leq \mu_{\mathcal{U}}(F,p) \sqrt{p_j} \quad \forall j \in \mathbb{N},
    \end{equation*}
    we will say that $p$ is \textbf{admissible}.
\end{definition}

An admissible distribution ensures that no indices with non zero coherence are ignored. The following lemma identifies the optimal sampling distribution that minimizes the required sample complexity.

\begin{lemma}\label{best-admissible}
    Given $F$ unitary in $\lnn$ and a finite union of convex cones $\mathcal{U}\subseteq\ell^2(\mathbb{N})$ contained in a finite-dimensional subspace, the minimum of $\mu_{\mathcal{U}}(F,p)$ over all admissible $p$ is achieved by $p^*$ where
    \begin{equation}\label{pstar}
        p^*_j = \frac{\alpha^2_j}{\|\alpha\|^2},
    \end{equation}
    in which case $\mu_{\mathcal{U}}(F,p^*)^2 = \|\alpha\|^2$.
\end{lemma}
We observe that the optimal distribution $p^*$ corresponds to sampling according to the normalized Christoffel function (often referred to as statistical leverage scores). This strategy is fundamental to optimal least squares approximation, as analyzed in \cite{cohen2017optimal} and recently reviewed in \cite{adcock2025optimal}.\smallskip

Our analysis often requires the assumption $\mu_\mathcal{U}(F,p) \geq 1$. As shown in Lemma \ref{coherence-energy-bound} below, this condition is mild; it effectively rules out pathological cases (e.g., vanishing coherence) and provides a practical criterion for numerical verification. Moreover, Lemma \ref{admiss>1} establishes that this condition holds automatically whenever $p$ is admissible.

\subsection{Recovery Guarantees}

We state our main recovery result. This theorem provides an error bound for the recovered signal $\hat{x}$. It distinguishes between the general case (where the sampling might ignore some active indices) and the ideal case where the distribution is admissible.

We adopt standard Minkowski notation for set operations; specifically, for a given set $A$, we denote its difference set as $A - A \coloneqq \{ a - b \mid a, b \in A \}$.

\begin{theorem}\label{thm:main-recovery-merged}
    Let $F$ be a unitary operator on $\ell^2(\mathbb{N})$, and let $p$ be a probability distribution on $\mathbb{N}$ with associated operators $S$, $D$, and $\hat{I}$. Let $G$ be a $(k, d)$-generalized generative network, and let $\mathcal{T} = R(G) - R(G)$ be its difference set, such that $\mu_{\mathcal{T}}(F,p) \geq 1$. If the number of measurements $m \in \mathbb{N}$ satisfies
    \begin{equation*}
        m \geq 16\mu_{\mathcal{T}}(F,p)^2  \left[ 2kd \log\left(\frac{2e k_d}{k}\right) + \log\left(\frac{4k}{\varepsilon}\right) \right],
    \end{equation*}
    then, with probability at least $1-\varepsilon$, for any $x_0 \in \ell^2(\mathbb{N})$, $b \coloneqq SDFx_0 + \eta$, $x^\perp \coloneqq  x_0 - \Pi_{R(G)}(x_0)$, and estimator $\hat{x} \in R(G)$ satisfying  $\lVert SDF\hat{x} - b \rVert \leq \min_{x\in R(G)} \lVert SDFx - b \rVert + \hat{\varepsilon}$, the following holds:
    \begin{equation}\label{equazione-madre}
        \lVert \hat{x} - x_0 \rVert \leq \lVert x^\perp \rVert + \sqrt{\frac{2}{m}} \Big[2 \lVert SDF x^\perp \rVert + 2\lVert\eta \rVert +\hat{\varepsilon} \Big]+ \left\|\hat{I}^\perp F(\hat{x}-\Pi_{R(G)}(x_0))\right\|.
    \end{equation}
    Furthermore, if $p$ is an \textbf{admissible} distribution, the final tail term vanishes and the bound simplifies to:
    \begin{equation}\label{bound-admissible-simp}
        \lVert \hat{x} - x_0 \rVert \leq \lVert x^\perp \rVert +\sqrt{\frac{2}{m}} \Big[2 \lVert SDF x^\perp \rVert + 2\lVert\eta \rVert +\hat{\varepsilon} \Big].
    \end{equation}
\end{theorem}

We observe that the best choice for the sampling probability distribution is the one outlined in Lemma \ref{best-admissible}, so if $\alpha$ is the local coherence of $F$ with respect to $\mathcal{T}$, then the bound becomes
\begin{equation}\notag
    m \geq 16\| \alpha\|^2  \left[ 2kd \log\left(\frac{2e k_d}{k}\right) + \log\left(\frac{4k}{\epsilon}\right) \right].
\end{equation}
When $p$ is admissible, the resulting bounds coincide with those known in the finite dimensional setting \cite{Berk22,berk2023model}. In contrast, the case of non admissible $p$ yields genuinely new bounds, representing a novel contribution in both the finite dimensional and infinite dimensional regimes.

\begin{remark}
    Due to the weighting operator $D$ in our measurement model $b \coloneqq SDFx_0 + \eta$, the noise contribution in the final bound scales as $\frac{1}{\sqrt{m}}\|\eta\|$. This provides a favorable asymptotic behavior, as the error vanishes for finite-energy noise, or remains strictly bounded for noise models scaling with $\sqrt{m}$.
\end{remark}

\subsubsection{The Balancing Property}
In many infinite dimensional scenarios, an admissible distribution (which requires $p_j > 0$ whenever $\alpha_j > 0$) might be impractical if the coherence has infinite support, or if we are constrained to sample only a finite subset of indices. In such cases, the term $\|\hat{I}^\perp F(\hat{x}-\Pi_{R(G)}\left(x_0\right))\|$ in Theorem \ref{thm:main-recovery-merged} does not vanish. To control this error, we introduce the \textit{balancing property}, a concept borrowed from infinite dimensional compressed sensing \cite{adcock2016generalized}.

\begin{definition}
    Given $F$ unitary in $\lnn$ , $p$ a probability distribution on $\mathbb{N}$, with $\hat{I}$ its associated operators. $G$ a $(k, d)$-generalized generative model. We say that the balancing property is verified for $\theta \in (0,1)$, if
    \begin{equation*}
        \| \hat{I}^\perp F \Pi_{R(G)}\| \leq \theta.
    \end{equation*}
\end{definition}

This property essentially requires that the energy of the signals in the model range is concentrated on the support of the sampling distribution $p$. If this holds, we can still derive a strong recovery bound, albeit with an amplification factor of $(1-\theta)^{-1}$.

\begin{remark} 
    One might question whether the balancing property is a realistic assumption. We note that for any desired tolerance $\theta > 0$, any unitary operator $F$, and any finite-dimensional generative model $G$, there exists a distribution $p$ such that the balancing property holds. To see this, consider a uniform distribution over the first $\ell$ indices, $p_\ell(n) = \frac{1}{\ell}\mathbbm{1}_{[\ell]}$, and let $\hat{I}_\ell$ be the corresponding operator projecting onto these first $\ell$ coordinates. Its orthogonal complement, $\hat{I}_\ell^\perp$, projects onto the remaining tail indices.

    Let $\mathcal{V}$ denote the finite-dimensional subspace containing the range $R(G)$, and let $R = F(\mathcal{V})$. We have:
    \begin{equation*}
        \| \hat{I}_\ell^\perp F \Pi_{\mathcal{V}} \| \leq \sup_{x\in R, \|x\|=1} \| \hat{I}_\ell^\perp x \|.
    \end{equation*}
    To bound this tail energy, let $(u_1, \dots, u_r)$ be an orthonormal basis for $R$. Any unit vector $x \in R$ can be expressed as $x = \sum_{i=1}^r c_i u_i$ with $\sum_{i=1}^r |c_i|^2 = 1$. By the Cauchy-Schwarz inequality, the $n$-th component satisfies $|x_n|^2 \leq \sum_{i=1}^r |(u_i)_n|^2$. Summing over the tail yields:
    \begin{equation*}
        \sup_{x\in R, \|x\|=1} \| \hat{I}_\ell^\perp x \|^2 = \sup_{x\in R, \|x\|=1} \sum_{n>\ell} |x_n|^2 \leq \sum_{n>\ell} \sum_{i=1}^r |(u_i)_n|^2.
    \end{equation*}
    Since each basis vector $u_i$ belongs to $\ell^2(\mathbb{N})$, the finite sum of their tail energies must strictly vanish as $\ell \to \infty$. Thus, $\| \hat{I}_\ell^\perp F \Pi_{\mathcal{V}}\| \to 0$. Therefore, by choosing a sufficiently large support size $\ell$, we can satisfy the balancing condition for any arbitrary $\theta$.
\end{remark}

\begin{corollary}\label{cor:balancing-bound}
    Given $F$ unitary in $\lnn$, $p$ a probability distribution on $\mathbb{N}$, $m \in \N$ with $S$ and $D$ its associated operators. Let  $G$ be a $(k, d)$-generalized generative network and $\mathcal{T} = R(G) - R(G)$ such that $\mu_{\mathcal{T}}(F,p)\geq 1$. Then if the balancing property is verified for $\theta \in (0,1)$ and
    \begin{equation*}
        m \geq 16 \mu_\mathcal{T}(F,p)^2  \left[ 2kd \log\left(\frac{2e k_d}{k}\right) + \log\left(\frac{4k}{\epsilon}\right) \right], 
    \end{equation*}
    the following holds with probability at least $1-\varepsilon$:\\
    for any $x_0 \in \lnn $, let $b \coloneqq SDFx_0 + \eta$ where $\eta \in \mathbb{C}^{m}$ and ,  $x^\perp \coloneqq  x_0 - \Pi_{R(G)}x_0$ . Let $\hat{x} \in R(G)$ satisfy $\lVert SDF\hat{x} -b \rVert \leq \min_{x\in R(G)} \lVert SDFx - b \rVert + \hat{\varepsilon}.$ Then
    \begin{equation*}
        \lVert \hat{x} - x_0 \rVert\leq \frac{1}{1-\theta}\left(\lVert x^\perp \rVert+\sqrt{\frac{2}{m}} \Big[2 \lVert SDF x^\perp \rVert + 2\lVert\eta \rVert +\hat{\varepsilon} \Big]\right)
    \end{equation*}
\end{corollary}
In other words, we retain the same form of stability guarantees as in the admissible case, at the cost of a worse constant, while using a sampling strategy that allows us to handle non admissible and thus more realistic probability distributions such as the one presented in the remark above.

\section{Numerical Validation}\label{sec:numerics}
The theoretical guarantees derived in the previous chapters are formulated within the infinite-dimensio\-nal Hilbert space $\ell^2(\N)$. While this continuum formulation provides rigorous bounds, direct numerical verification in the exact continuum limit is computationally infeasible. Consequently, any numerical validation relies on finite dimensional discretizations. To ensure that our observations are not merely artifacts of specific grid choices, but rather accurate reflections of the underlying function-space theory, it is crucial to employ an architecture that operates consistently across varying resolutions.
To this end, we leverage the \textit{Functional Autoencoder} (FAE) framework introduced by \cite{bunker2025autoencoders}. Unlike standard discrete autoencoders, FAEs are designed to be natively mesh free, learning operators that map between function spaces rather than fixed Euclidean vectors. By adopting this architecture, we can empirically demonstrate the central claim of this work: that the proposed recovery guarantees and adaptive sampling strategies are robust to the discretization dimension.
In the following sections, we apply this framework to the Darcy flow equation, a classic application for scientific machine learning. We perform experiments across varying spatial resolutions, ranging from $32 \times 32$ to $128 \times 128$, to verify that the reconstruction error and local coherence estimates exhibit asymptotic behavior consistent with the theoretical infinite dimensional limit, independent of the specific mesh granularity used during training.
\subsection{Experimental Setup and Methodology}
The following sections detail the numerical validation of our theoretical framework, specifically addressing the inverse problem associated with the steady state Darcy flow equation \cite{darcy1} \cite{darcy2}. This governing equation, defined as:
\begin{equation}
    -\nabla \cdot (a(x) \nabla u(x)) = f(x), \quad x \in \Omega,
\end{equation}
describes fluid flow through a porous medium on a bounded domain $\Omega \in \R^2$, where $a(x)$ represents the permeability coefficient of the medium, $u(x)$ denotes the resulting pressure field (the solution), and $f(x)$ is the source term.\smallskip 

Our objective is to recover the pressure field $u \in L^2(\Omega)$ from a finite set of noisy Fourier measurements. Formally, we consider the inverse problem
$$b = S D \mathcal{F} u + \eta,$$
where $\mathcal{F}$ denotes the discrete Fourier transform at a high-fidelity resolution (fixed at $128 \times 128$ in our experiments), $S$ is a sampling operator that selects a subset of $m$ frequencies based on a specific probability distribution (uniform or coherence-based), $D$ is the associated weighting operator to ensure unbiasedness, and $\eta$ represents measurement noise. Since the noise dependence was not the main focus of this work, all the experiments are performed with noiseless measurements. We approach this ill-posed problem by restricting the solution to the range of a FAE, $G: \mathbb{R}^k \to L^2(\Omega)$, trained to approximate the solution manifold of the Darcy flow. The reconstruction is performed by searching for the optimal latent code $\hat{z}$ that minimizes the data fidelity loss:
$$\hat{z} = \underset{z \in \mathbb{R}^k}{\text{argmin}} \| S D \mathcal{F} (U_r G_{r}(z)) - b \|_2^2,$$
where $G_{r}$ is the generator trained at resolution $r$, and $U_r$ is a fixed upscaling operator mapping the generator's output to the measurement resolution. This formulation allows us to decouple the generation resolution from the measurement resolution, testing the resolution-invariance of our recovery guarantees.

\subsubsection{Dataset and Architecture}
We validate our framework using the Darcy flow dataset provided in \cite{darcy-dataset}, whose specific implementation details are given in \cite[Appendix A.3.2]{darcy-dataset}. These solutions exhibit complex, multi scale features that challenge standard linear reconstruction methods. The original dataset contains 2048 images at 421 x 421 resolution, that for out experiments we resize to 32 x 32, 64 x 64 and 128 x 128.

We utilize the FAE introduced in \cite{bunker2025autoencoders}. However, a critical modification is made to the architecture: we replace the Gaussian Error Linear Unit (GeLU) activation functions with Rectified Linear Units (ReLU). This modification is necessary to ensure the network is within the theoretical guarantees developed in this work. Apart from this substitution, the rest of the network is kept identical to the original implementation and satisfies the assumptions required for our theoretical analysis.

\subsubsection{Training}
To demonstrate the resolution independent capabilities of the FAE, the training is carried out at three distinct spatial resolutions: $32\times32$, $64\times64$, and $128\times128$. The FAE training is performed in an entirely unsupervised manner. The network, consisting of an encoder $\mathcal{E}$ and a decoder $\mathcal{D}$, is optimized to compress the input function into a low dimensional latent representation $z \in \mathbb{R}^k$ (in our experiments $k=64$) and subsequently decompress it back to the original function domain. The objective is to minimize the reconstruction error while learning a discretization invariant representation of the physical operator. We adopt the same training protocol described in \cite{bunker2025autoencoders}, and specifically, we train our FAE for 50,000 epochs, with a learning rate equal to 0.001 and a batch size of 32.
\subsubsection{Fitting the Latent Space}
Once the autoencoder training is completed, the decoder can be utilized independently as a generator $G(z)$ once we can sample $z$ from the latent distribution. However the aggregated posterior distribution of the encoded training data often does not match a standard Gaussian as in variational autoencoders \cite{kingma2013auto}. Consequently, simple Gaussian sampling would likely yield latent codes that map to out of distribution functions.
To accurately capture the non-Gaussian structure of the learned latent manifold, we fit a Gaussian Mixture Model (GMM) with $K=10$ components to the latent codes of the training data:
\begin{equation}\label{prior}
    p(z) = \sum_{i=1}^{K} \pi_i \mathcal{N}(z|\mu_i, \Sigma_i).
\end{equation}
This GMM serves two critical roles: it provides a valid generative mechanism for sampling, and it acts as a structured prior during the recovery phase, allowing us to penalize solutions that drift away from the support of the training distribution.
\subsubsection{Reconstruction}
The recovery problem is formulated as finding the latent space element $z$ that best matches the observed sub-sampled Fourier measurements.
A crucial aspect of our implementation is the handling of multi-resolution generators against a fixed high-resolution ground truth that we fix at $128 \times 128$. Let $u_{\text{target}}$ be the ground truth signal at that resolution. For a generator $G_r$ trained at resolution $r$ (where $r \in \{32, 64, 128\}$), we solve the following optimization problem:
\[ 
\hat{z} = \argmin_{z \in \mathbb{R}^k} \| SD \mathcal{F} ( U^{128}_r G_r ( z ) )  - SD\mathcal{F} (u_{\text{target}} ) \|^2 
\]
where $U^{128}_r$ denotes a bicubic upscaling operator that maps the generator output to the target resolution, and $\mathcal F$, $S$, and $D$ are the equivalent of the operators introduced in section \ref{settupp}, with the difference that now we have to operate in finite dimension, so $\mathcal F$ is the FFT in dimension $128\times 128$, $S$ is the sampling operator that samples $m$ Fourier coefficient based on the probability we have chosen (in our experiments either the one induced by the local coherence or uniform), and finally $D$ weights the samples in the same manner.
This formulation ensures that all models are evaluated against the same high-fidelity physical ground truth regardless of their internal resolution. We also ensure that the sampling operator is the same across different output dimensions since everything is up-scaled before measurement via the Fourier transform and sampling are performed. The optimization is performed using the Adam optimizer with a learning rate of $0.01$ for 500 iterations. We employ an implementation that optimizes latent codes for multiple test instances simultaneously. We evaluate performance using the Mean Squared Error (MSE) on the spatial domain as the metric.

\subsection{Results}
We now present the results of our experiments. We used half of our datasets, so 1024 images, at the 3 different, for the training of the general model as well to fit the GMM model to the latent space. The other half of the images are used to compute their reconstruction from the undersampled Fourier transform. All the numerical results that follows are means across all the images of this test set.

\subsubsection{Local coherence and adaptive sampling}

To compute the optimal sampling distribution~\eqref{pstar}, one would need to evaluate the local coherence (Definition~\ref{local-coherence}). However, this quantity is not explicitly computable, so we approximate it via a Monte Carlo procedure.

Specifically, we generate $N = 1024$ samples from the trained GMM, pass them through the generator, and upscale the outputs to the target resolution. We then compute the maximum magnitude of the corresponding Fourier coefficients across these samples, obtaining
\[
\alpha^r_j = \max_{i = 1,\dots,N} 
\left| \bigl(\mathcal{F} U^{128}_r (G_r(z_i))\bigr)_j \right|,
\]
where $r \in \{32, 64, 128\}$ and $j$ indexes the Fourier frequencies. The sampling probabilities $p_j$ are then obtained by normalizing $\alpha^r_j$. The dependence on $r$ arises because generators trained at different resolutions exhibit different local coherence structures. We compare this adaptive sampling strategy against a uniform sampling baseline.

From a theoretical perspective, local coherence should depend on normalized differences of signals in the model range rather than on individual samples. Motivated by this observation, we also tested an alternative Monte Carlo estimator based on self-differences. Concretely, given a batch of generated images $\{u_i\}_{i=1}^B$, we computed all pairwise differences $u_i - u_j$, normalized each difference, applied a channel-wise two-dimensional Fourier transform, and then maximized the modulus across the batch of normalized differences. The resulting coherence proxy was squared and normalized to obtain sampling probabilities.

While this estimator is more closely aligned with the theoretical definition of local coherence, it consistently underperformed the simpler maximum magnitude estimator described above. Empirically, the original strategy based on maximizing the Fourier magnitude of individual generated samples led to a sampling distributions with improved reconstruction performance.

Figure~\ref{fig:coherence-comparison} compares the two approaches with the naive approach of uniform sampling. Despite its stronger theoretical motivation, the self-difference estimator does not translate into improved empirical performance in our setting.

\begin{figure}[t]
    \centering
    \includegraphics[width=0.9\linewidth]{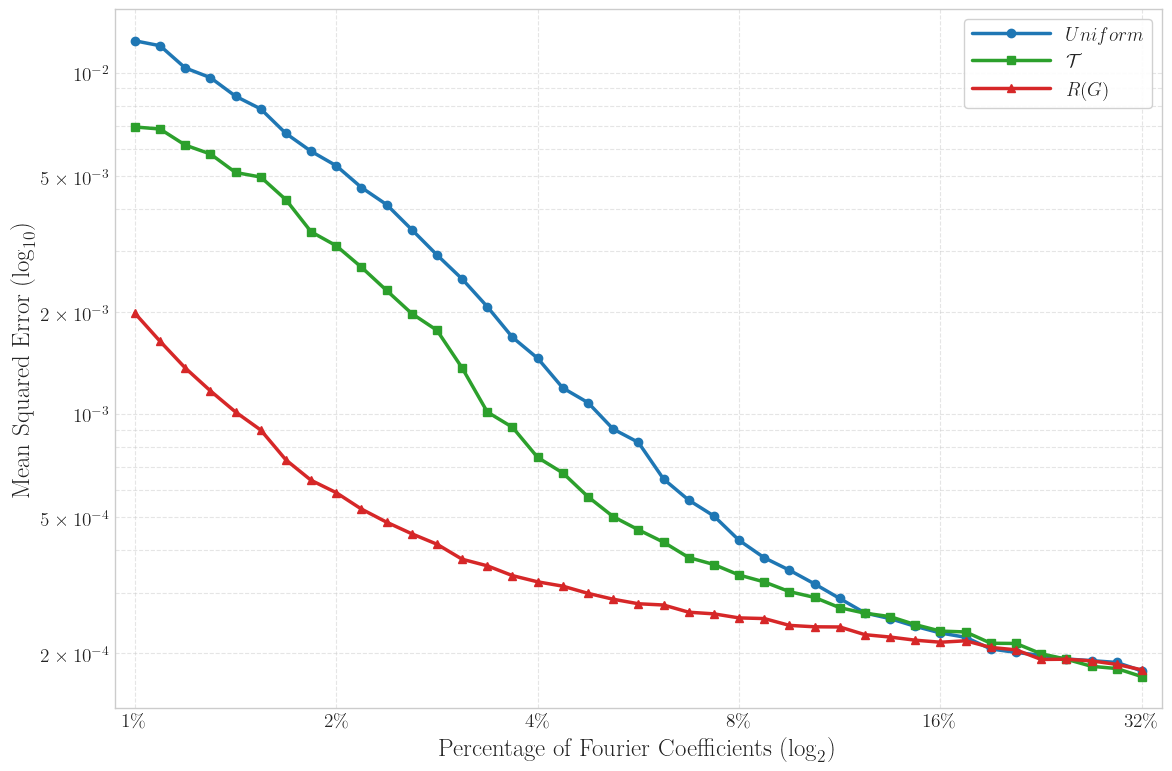}
    \caption{Comparison between the original maximum-magnitude based coherence estimator and the self-difference based estimator. Although the latter is theoretically better aligned with the definition of local coherence, the former yields superior empirical performance (model trained and with output in 64 x 64 resolution).}
    \label{fig:coherence-comparison}
\end{figure}

\subsubsection{Adaptive sampling}
To evaluate the reconstruction performance, we vary the number of measurements $m$ on a logarithmic scale, ranging from $1\%$ to $32\%$ of the total number of Fourier coefficients. This allows us to analyze the asymptotic behavior of the recovery error as the sampling density increases. 
As shown in Figure~\ref{line-au}, adaptive sampling significantly outperforms the uniform baseline in the undersampled regime, exhibiting superior efficiency for a fixed error threshold. As the sampling rate increases, the two strategies asymptotically converge to comparable error levels. Moreover, this behavior is consistent across models trained on data from all three discretizations considered. In Figure~\ref{img-au}, we provide a qualitative comparison of the reconstruction results at different sampling rates.

\begin{figure}[!htb]
    \centering
    \includegraphics[width=0.9\linewidth]{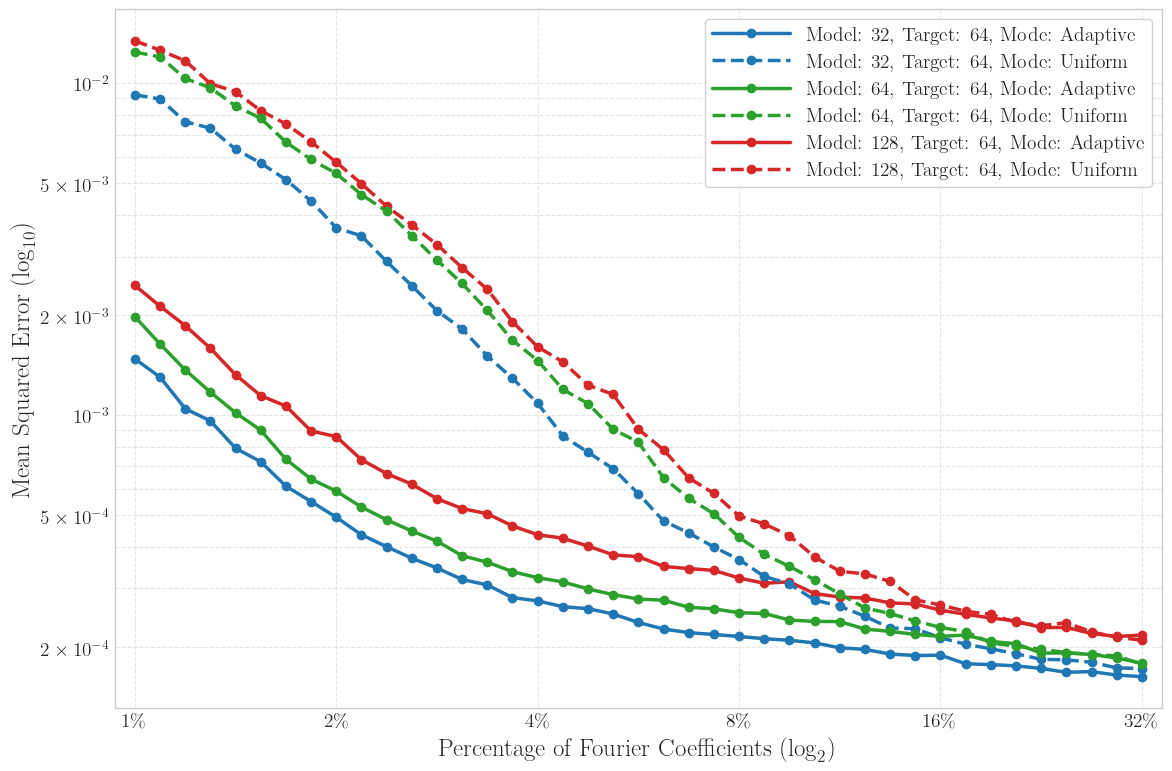}
\caption{Reconstruction error versus sampling rate for uniform and adaptive Fourier sampling.\label{line-au}}
\end{figure}
\begin{figure}[!htb]
    \centering
    \includegraphics[width=0.7\linewidth]{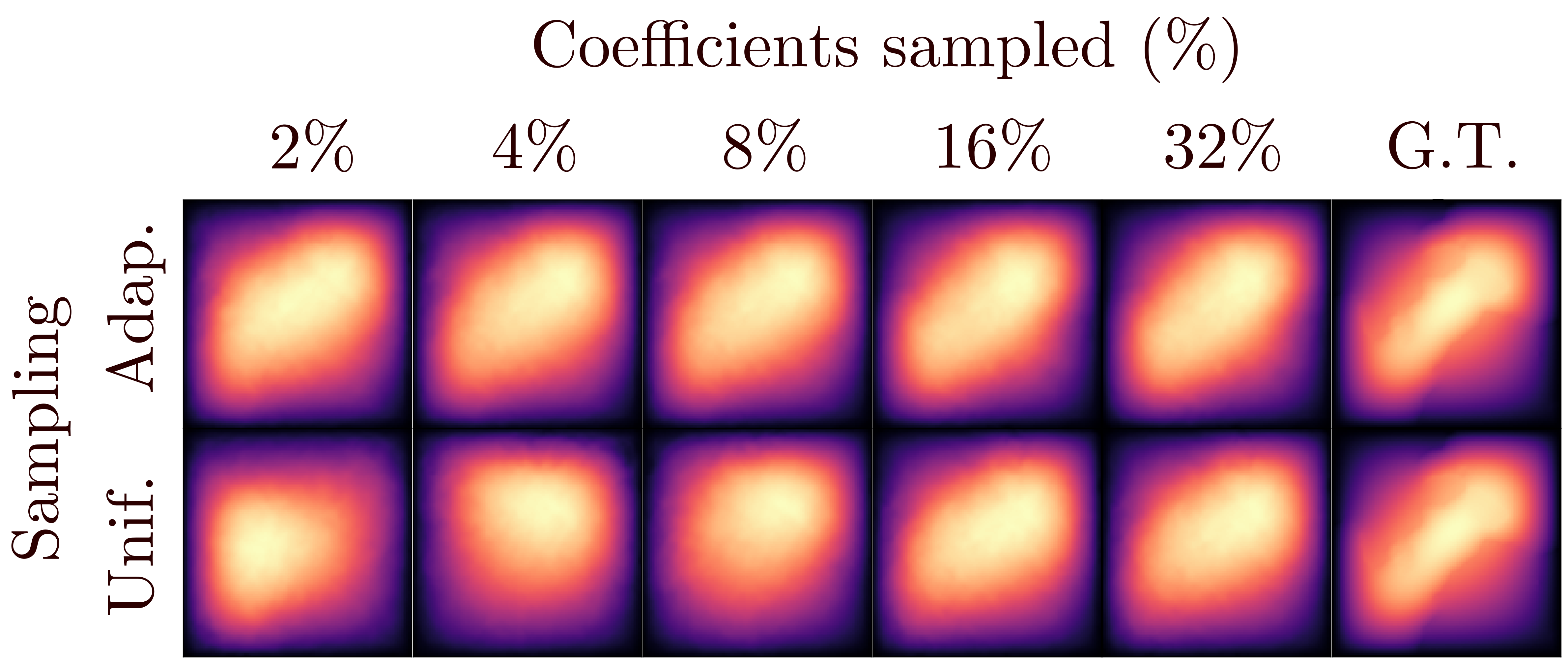}
\caption{Reconstruction examples comparing uniform and adaptive sampling at different rates (model trained on 64 x 64 resolution).\label{img-au}}
\end{figure}

\subsubsection{Discretization invariance of the reconstruction error}
The main objective of this experiment is to show that the reconstruction performance is mostly invariant with respect to the output dimension. In Figure~\ref{linee-dim}, we report the recovery error for three different models, each trained at one of the three discretization levels considered, and evaluated while generating outputs at three different resolutions (namely $32 \times 32$, $64 \times 64$, and $128 \times 128$). 

We observe that, in the asymptotic regime where a sufficient number of measurements is available, the recovery error is primarily determined by the training resolution of the model, and is essentially independent of the target output dimension. In contrast, in the severely undersampled regime, the model resolution has a mild impact on performance, with lower output dimensions leading to slightly better reconstructions. 

Figure~\ref{img-dim} qualitatively illustrates the reconstruction performance across varying output dimensions. We observe that in the sparse sampling regime, the recovery is dominated by low frequency features, which naturally favors models generating at lower resolutions. Furthermore, in the regime of abundant samples, models trained on lower resolutions exhibit superior performance. This behavior can be understood through the lens of the spectral bias of neural networks \cite{rahaman2019spectral}, which shows that deep networks prioritize learning low-frequency functions. In the severely undersampled regime, the available measurements primarily constrain the low-frequency components of the signal. A high-resolution generator, having a larger capacity, is prone to 'hallucinating' high-frequency artifacts that lie in the null space of the measurement operator. In contrast, a lower-resolution generator acts as an explicit low-pass filter, effectively regularizing the inverse problem by suppressing these spurious high-frequency components and ensuring the recovery aligns with the spectral bias inherent to neural networks.
\begin{figure}[!htb]
    \centering
    \includegraphics[width=0.9\linewidth]{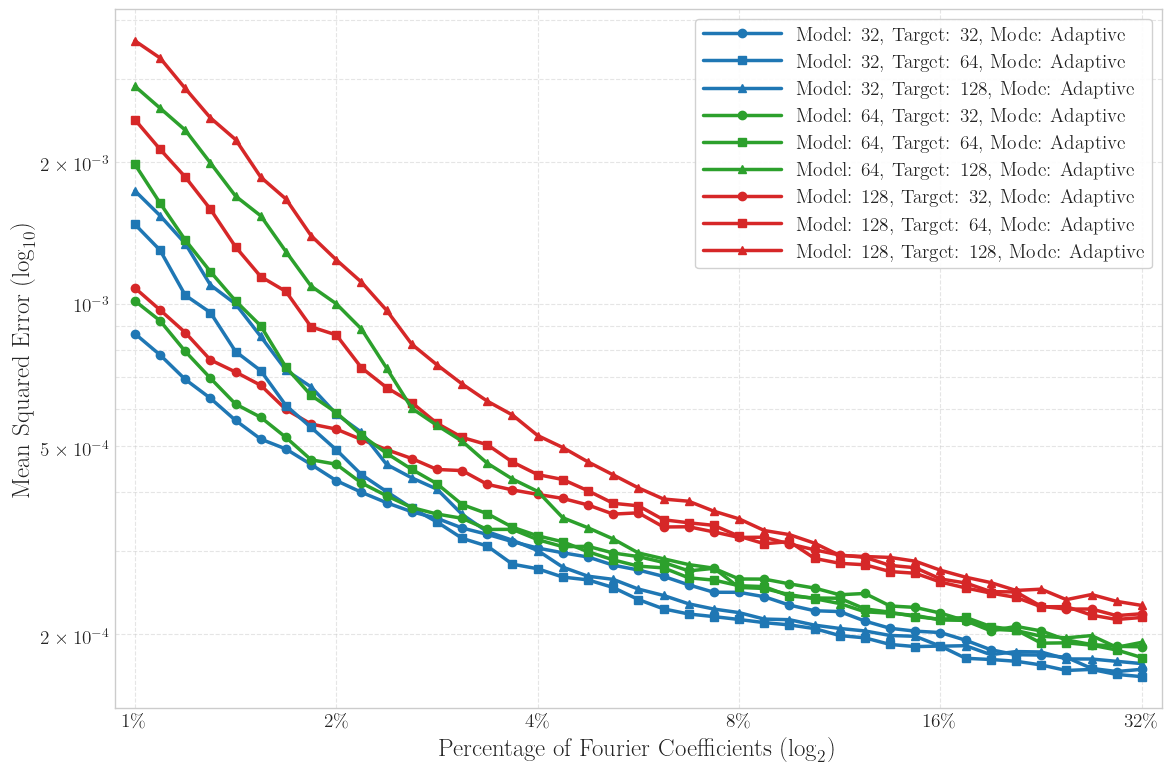}
\caption{Recovery error across different training and output discretizations.\label{linee-dim}}
\end{figure}
\begin{figure}[!htb]
    \centering
    \includegraphics[width=0.7\linewidth]{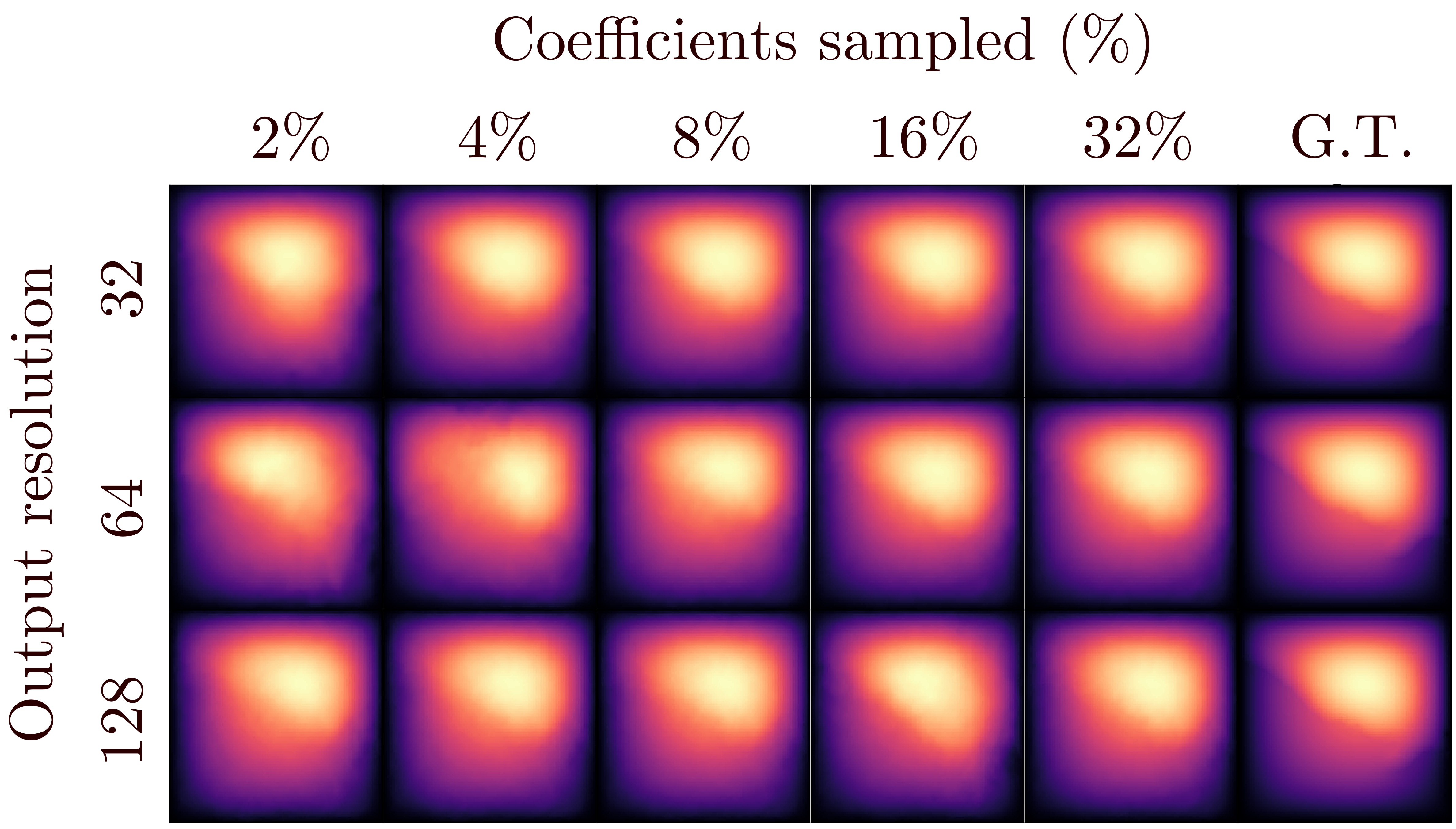}
\caption{Example reconstructions for varying output resolutions (model trained on 64 x 64 resolution)\label{img-dim}}
\end{figure}

\subsubsection{GMM as regularizer}
In this experiment, we investigate the effect of incorporating the latent prior~\eqref{prior} as an explicit regularization term in the reconstruction optimization. Instead of solving a purely data fidelity problem, we augment the objective with a penalty that encourages latent variables to remain in regions of high probability under the fitted GMM. This promotes solutions that are more consistent with the latent distribution learned during training. Unlike the previous experiments, this study is not concerned with resolution invariance.

The resulting optimization problem can be written as
\[
\hat{z}
=
\argmin_{z \in \mathbb{R}^k}
\Bigl\|
SD \mathcal{F}\bigl(U^{128}_r G_r(z)\bigr)
-
SD \mathcal{F}(u_{\text{target}})
\Bigr\|^2
+
\lambda \, \bigl(-\log p(z)\bigr),
\]
where $p(z)$ is the GMM prior defined in~\eqref{prior} and $\lambda \geq 0$ controls the strength of the regularization.

Rather than fixing $\lambda$ a priori, we perform a parameter-tuning procedure to identify the most effective regularization strength by using $10\%$ of our dataset as validation. Specifically, for a discrete set of candidate values of $\lambda$, we compute the mean reconstruction error averaged across all sampling levels $m$. The value of $\lambda$ that minimizes this aggregated error is then selected as the optimal regularization parameter, in our case we found $\lambda = 10^{-4}$. The performance obtained with this tuned regularizer is subsequently compared against the unregularized case ($\lambda = 0$) for the remaining data points.

Figure~\ref{regul} reports the reconstruction performance using the optimally tuned regularization parameter as a function of the number of measurements. In the severely undersampled regime, the inclusion of the GMM prior leads to a clear improvement over the unregularized approach, as the prior effectively constrains the optimization to regions of the latent space corresponding to plausible generator outputs. In this context, plausible or realistic refers to samples that are likely under the generator induced data distribution, rather than to any external notion of physical realism.

As the number of measurements increases, the data fidelity term becomes increasingly informative, and the advantage of regularization progressively diminishes. In the high sampling regime, the unregularized formulation performs better than the regularized one, indicating that the measurements alone are sufficient to accurately identify the latent code without introducing bias from the prior. Overall, these results highlight the complementary roles of the measurement operator and the latent prior, with the latter being most beneficial when observational information is limited.

\begin{figure}[!htb]
    \centering
    \includegraphics[width=0.9\linewidth]{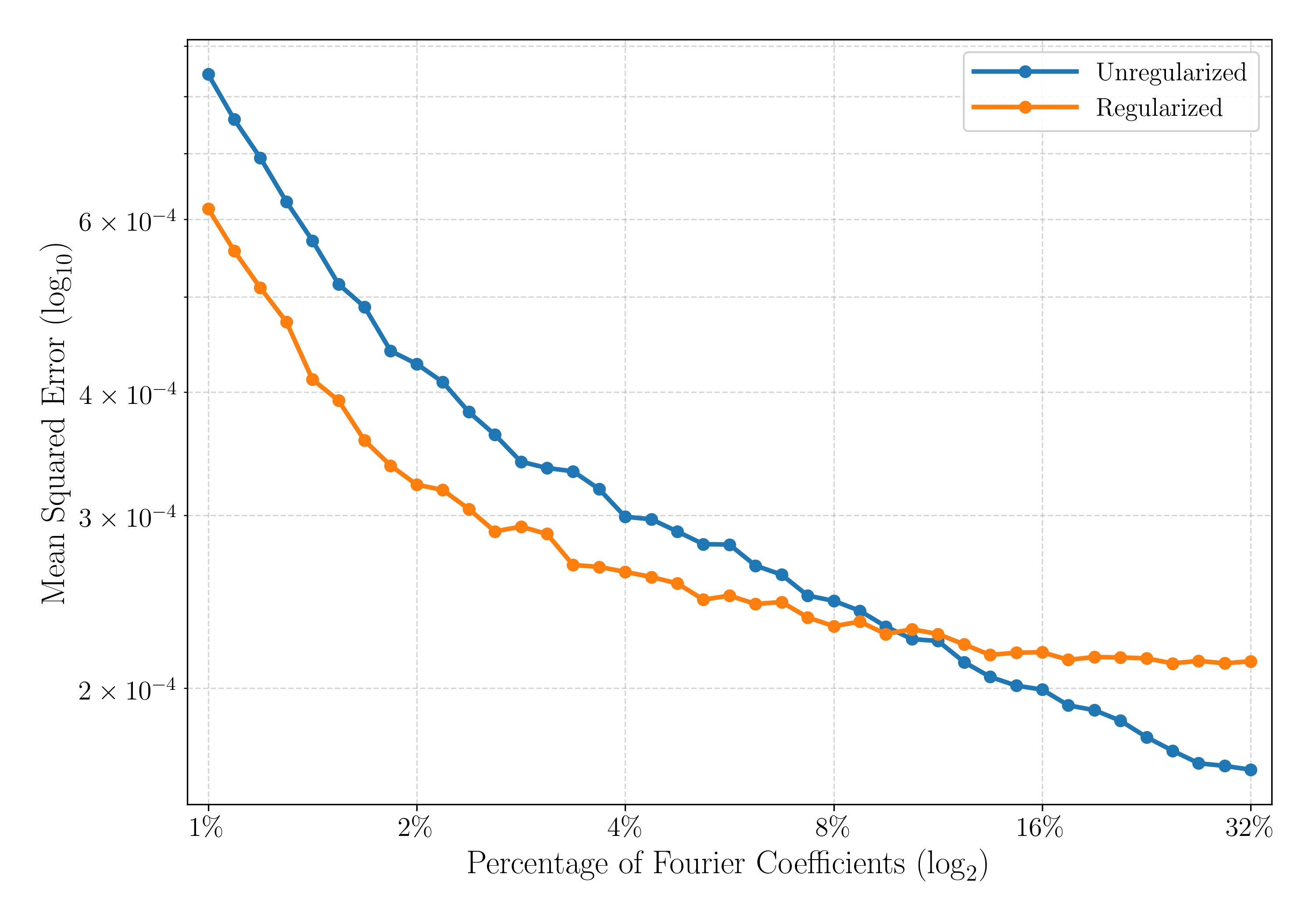}
\caption{Comparison between our proposed method and classical compressed sensing (CS).\label{regul}}
\end{figure}

\section{Proofs}\label{sec:proofs}
We now provide all the results and proofs needed for the results of chapter 2
\subsection{Setup and local coherence}
Let us start by characterizing the range of a $(k,d)$-Genera\-lized Generative Network, in particular obtaining that its range is contained in a finite dimensional space.
\begin{lemma}\label{lemma-sottospazi}
    Let $2\leq k \coloneqq k_0 \leq ... \leq k_d$ and $G$ be a $(k,d)$-Generalized Generative Network. Then $R(G)$ is a union of no more than $N$ at-most $k$ dimensional polyhedral cones where
    \begin{equation*}
        N \leq \left( \frac{2ek_d}{k} \right)^{kd}.
    \end{equation*}
\end{lemma}
     \begin{proof}
        The proof follows the counting argument of \cite[Lemma S2.2]{Berk22}.  We observe that the geometric mean of the layer widths is bounded by the maximum width $k_d$. Therefore, the range of the penultimate layer, $\tilde{G}(z) = \sigma(W^{(d)}\dots\sigma(W^{(1)}z))$, resides in $\mathbb{R}^{k_d}$ and is a union of at most $N\leq\left( \frac{2ek_d}{k} \right)^{kd}$ polyhedral cones. The final layer of our generator is a linear operator $W: \mathbb{C}^{k_d} \to l^2(\mathbb{N})$. Since $W$ is a linear map from a finite-dimensional space, it is necessarily bounded. A bounded linear operator maps a polyhedral cone in the domain to a polyhedral cone in the codomain (specifically, the image is the conic hull of the images of the generators). Thus, the geometric structure of the union of cones is preserved in $l^2(\mathbb{N})$ without increasing the number of components $N$.
    \end{proof}
\begin{remark}\label{remark-T}
    Since $R(G)$ is a union of at most $N$ $k$-dimensional cones we observe that the Minkowski sum $\mathcal{T} = R(G) - R(G) $ is a union of at most $N^2$ $2k$-dimensional cones.
\end{remark}
    
    Let us now prove that local coherences are well defined.
    \begin{proof}[Proof of Lemma \ref{alpha-in-l2}]
        Let $(\varphi_i)_{i\in[r]}$ be a basis of $\Delta(\mathcal{U})$ and $(F_j := F^* e_j)_{j\in\mathbb{N}}$ the row of $F$ which form an orthonormal basis of $\lnn$ since $F$ is unitary. Then we can consider
        \begin{equation*}
            \alpha_j = \sup_{c\in\mathbb{C}^r \colon \|c\|=1} |\langle F_j , \sum_{i\in[r]}c_i\varphi_i\rangle| = |\langle F_j , \sum_{i\in[r]}\bar{c}_{i,j}\varphi_i\rangle|,
        \end{equation*}
        where $\bar{c}_{i,j}$ are the coefficients attaining the supremum (which is achieved as a maximum by compactness). Then we can conclude with
        \begin{align*}
            \sum_{j\in\mathbb{N}} \alpha_j^2 & = \sum_{j\in\mathbb{N}} \langle F_j , \sum_{i\in[r]}\bar{c}_{i,j}\varphi_i\rangle^2 \\
            &= \sum_{j\in\mathbb{N}} \left(\sum_{i\in[r]}\bar{c}_{i,j}\langle F_j , \varphi_i\rangle\right)^2 \\
            & \leq r\sum_{j\in\mathbb{N}} \sum_{i\in[r]}\bar{c}_{i,j}^2\langle F_j , \varphi_i\rangle^2 \\
            & \leq r\sum_{j\in\mathbb{N}} \sum_{i\in[r]}\langle F_j , \varphi_i\rangle^2 \\
            &= r\sum_{i\in[r]}\sum_{j\in\mathbb{N}}\langle F_j , \varphi_i\rangle^2 \\
            &= r\sum_{i\in[r]} \| \varphi_i \|^2 = r^2 <+\infty. \qedhere
        \end{align*}
    \end{proof}
The following two lemmas establish conditions under which $\mu_\mathcal{U}(F,p) \geq 1$, which are essential for our main results. In particular, the second lemma guarantees that this condition holds whenever $p$ is admissible.

\begin{lemma}\label{non-admiss>1}
    Let $F$ be a unitary operator in $\lnn$, $\mathcal{U} \subseteq \lnn$ be a finite union of convex cones contained in a finite-dimensional subspace, and $p$ be a probability distribution on $\mathbb{N}$. If there exists a unit vector $\hat{u} \in \Delta(\mathcal{U})$ such that $\| \hat{I} F\hat{u} \| = 1$, then $\mu_\mathcal{U}(F,p) \geq 1$.
\end{lemma}
\begin{proof}
    Consider $\beta_j = (F\hat{u})_j$. We have that $\mu_{\mathcal{U}}(F,p)\geq B$ for the smallest $B$ that verifies
    \begin{equation*}
        \beta_j \leq B\sqrt{p_j} \quad \forall j\colon p_j>0,
    \end{equation*}
    we can take the squared norm on both sides to obtain
    \begin{equation*}
        \| \hat{I} F\hat{u} \| \leq B,
    \end{equation*}
    since we look for the smallest constant that verifies the inequality and $\| \hat{I} F\hat{u} \|=1$ we conclude that $B=1$ and $\mu_{\mathcal{U}}(F,p) \geq B =1$
\end{proof}

\begin{lemma}\label{admiss>1}
   Let $F$ be a unitary operator on $\lnn$, and let $\mathcal{U} \subseteq \lnn$ be a finite union of convex cones contained in a finite-dimensional subspace. If $p$ is an admissible probability distribution on $\mathbb{N}$, then$$\mu_{\mathcal{U}}(F,p) \geq 1.$$
\end{lemma}
\begin{proof}
    By Lemma~\ref{non-admiss>1}, we only need to show there exists $\hat{u} \in \Delta(\mathcal{U})$ with $\| \hat{u} \|=1$ such that $\| \hat{I}F\hat{u} \| = 1$, but this holds for all of them. If $(F\hat{u})_j \neq 0$, then by definition $\alpha_j > 0$. The admissibility of $p$ guarantees that $p_j > 0$ for all such $j$, meaning the projection $\hat{I}$ acts as the identity on $F\hat{u}$. Since $F$ is a unitary operator, we conclude $\|\hat{I}F\hat{u}\| = \|F\hat{u}\| = \|\hat{u}\| = 1$. 
\end{proof}

Intuitively, to guarantee $\mu_{\mathcal{U}}(F,p) \geq 1$, it suffices for the sampling distribution to place its support on a set of coordinates whose cumulative squared coherence is at least $1$. This ensures the measurement operator provides sufficient coverage of the generator's range, as formalized in the following lemma.

\begin{lemma}\label{coherence-energy-bound}
    Given $F$ unitary in $\lnn$, $\mathcal{U}\subseteq\lnn$ a finite union of convex cones contained in a finite-dimensional subspace, and a probability distribution $p$ with $\hat{I}$ the projection on the support of $p$. If we have 
    \begin{equation}\label{magg-1}
        \| \hat{I} \alpha \| \geq 1,
    \end{equation}
    then $\mu_{\mathcal{U}}(F,p) \geq 1$.
\end{lemma}

\begin{proof}
  
From the definition of $ \mu_{\mathcal{U}}(F,p) $, if we denote with $\Omega$ the support of $p$ we have
    \begin{equation*}
        \alpha_j^2 \leq \mu_{\mathcal{U}}(F,p)^2 p_j \quad \forall j \in \Omega.
    \end{equation*}
    Summing this inequality over all indices in $\Omega$:
    \begin{equation*}
        \sum_{j \in \Omega} \alpha_j^2 \leq \sum_{j \in \Omega} \mu_{\mathcal{U}}(F,p)^2 p_j = \mu_{\mathcal{U}}(F,p)^2 \sum_{j \in \Omega} p_j.
    \end{equation*}
    Therefore,
    \begin{equation*}
        \sum_{j \in \Omega} \alpha_j^2 \leq \mu_{\mathcal{U}}(F,p)^2.
    \end{equation*}
    Taking the square root implies
    \begin{equation*}
        \mu_{\mathcal{U}}(F,p) \geq \sqrt{\sum_{j \in \Omega} \alpha_j^2}.
    \end{equation*}
    Thus, if \eqref{magg-1} holds, it follows directly that $\mu_{\mathcal{U}}(F,p) \geq 1$.
\end{proof}

Finally we prove how we can characterize the best probability distribution for signal recovery using coherence.
\begin{proof}[Proof of Lemma \ref{best-admissible}]
    We simply have to observe that $\alpha_j^2 \leq \mu_{\mathcal{U}}(F,p)^2 p_j$ for all $j\in\N$ and then compute
    \begin{equation*}
        \sum_{j\in\mathbb{N}} \alpha_j^2 \leq \sum_{j\in\mathbb{N}} \mu_{\mathcal{U}}(F,p)^2 p_j
    \end{equation*}
    which implies $\|\alpha\|^2 \leq \mu_{\mathcal{U}}(F,p)^2 $.
\end{proof}

\subsection{Generalized Restricted Isometry Property}
We now present two results establishing that our sensing operator preserves the geometry of the signal set. This behavior is analogous to the Restricted Isometry Property (RIP) in classical compressed sensing; accordingly, we refer to it as the Generalized Restricted Isometry Property (Gen-RIP). 

\begin{definition}[Generalized Restricted Isometry Property]\label{def:gen-rip}
Let $\mathcal{K} \subseteq \lnn$ be a signal set of interest and let $\mathcal{F}$ be a unitary operator and $p$ a probability distribution on $\mathbb{N}$ with associated operators $S$, $D$, and $\hat{I}$. We say that $\mathcal{A}$ satisfies the Generalized Restricted Isometry Property (Gen-RIP) over $\mathcal{K}$ with constant $\delta \in (0, 1)$ if
\[
\sup _{{x} \in \mathcal{K} , \|x\|=1}\left|\frac{1}{m}\|S D F {x}\|^2-\| \hat{I} F {x}\|^2\right| \leq \delta 
\]
\end{definition}

To establish this property for our sensing operator, we begin with an auxiliary result.

\begin{lemma}[Matrix Bernstein { \cite[Theorem 5.4.1]{bern}}]\label{bern}
Let $X_1, \dots, X_N$ be independent, mean zero, $n \times n$ symmetric random matrices, such that $\|X_i\| \le K$ almost surely for all $i$. Then, for every $t \ge 0$, we have
\[
\mathbb{P} \left\{ \left\| \sum_{i=1}^N X_i \right\| \ge t \right\} \le 2n \exp \left( - \frac{t^2/2}{\sigma^2 + Kt/3} \right),
\]
\textit{where $\sigma^2 = \left\| \sum_{i=1}^N \mathbb{E} X_i^2 \right\|$.}
\end{lemma}
Now we give the first results that applies to general sets $\mathcal{U}$.
\begin{lemma}\label{lemma-rip-gen}
    Let $F$ be a unitary operator on $\lnn$, and let $\mathcal{U} \subseteq \lnn$ be a $r$-dimensional subspace. Let $p$ be a probability distribution on $\mathbb{N}$ with associated operators $S$, $D$, and $\hat{I}$, such that $\mu_{\mathcal{U}}(F,p) \geq 1$. Then, for $\varepsilon\in(0,1)$ and $\delta \in (0,1]$, we have that
    \begin{equation*}
        \mathbb{P} \left[ \sup _{{x} \in \mathcal{U} , \|x\|=1}\left|\frac{1}{m}\|S D F {x}\|^2-\| \hat{I} F {x}\|^2\right| \leq \delta \right] \geq 1-\varepsilon
    \end{equation*}
    provided
    \begin{equation*}\label{rate}
        m \geq \frac{4\mu_{\mathcal{U}}(F,p)^2}{\delta^2}\log\left( \frac{2r}{\varepsilon}\right).
    \end{equation*}
    \begin{proof}
        Since $\mathcal{U} \subset \lnn$ is an $r$-dimensional subspace, we can fix an orthonormal basis $(\psi_i)_{i\in[r]}$ for it. We then define the operator $P_\mathcal{U}\colon \lnn \to \mathbb{C}^r$ by
        \begin{equation*}
            (P_\mathcal{U}(x))_i = \langle x , \psi_i \rangle
        \end{equation*}
        We observe that $P_\mathcal{U}^* P_\mathcal{U} = \Pi_\mathcal{U}$, i.e. $x=P_\mathcal{U}^* P_\mathcal{U} x$ for all $x \in \mathcal{U}$. Let us now compute
        \begin{align*}
            &\sup _{{x} \in \mathcal{U}, \|x\|=1}\left|\frac{1}{m}\|SDF{x}\|^2- \| \hat{I} F {x}\|^2 \right|\\
            &\quad =\sup _{{x} \in \mathcal{U} , \|x\|=1}\left|\frac{1}{m}\left\|SDFP_{\mathcal{U}}^* P_{\mathcal{U}} {x}\right\|^2- \| \hat{I} F P_{\mathcal{U}}^* P_{\mathcal{U}} {x}\|^2\right| \\
            &\quad  =\sup _{{u} \in \mathbb{C}^r \cap \mathbb{S}^{r-1}}\left|\frac{1}{m}\left\|SDFP_{\mathcal{U}}^* {u}\right\|^2-\| \hat{I} F P_{\mathcal{U}}^* {u}\|^2\right| \\
            &\quad  =\frac{1}{m} \sup _{{u} \in \mathbb{C}^r \cap \mathbb{S}^{r-1}}\left|{u}^*\left[\left(SDFP_{\mathcal{U}}^*\right)^*\left(SDFP_{\mathcal{U}}^*\right)-m \left( \hat{I} F P_{\mathcal{U}}^*\right)^*\left( \hat{I} F P_{\mathcal{U}}^*\right)\right] {u}\right|.
        \end{align*}
        We can write this as the operator norm
        \begin{equation*}\label{summ}
            \frac{1}{m}\left\|P_{\mathcal{U}} F^* D S^* S D F P_{\mathcal{U}}^*-m P_{\mathcal{U}} F^*\hat{I} F P_{\mathcal{U}}^*\right\| = \frac{1}{m}\left\|\sum_{i=1}^m\left[P_{\mathcal{U}} F^* \left[ D s_i s_i^* D - \hat{I} \right]F P_{\mathcal{U}}^*\right]\right\|.
        \end{equation*}
        First we observe that
        \begin{equation*}
            \mathbb{E}[D s_i s_i^* D] = \hat{I}.
        \end{equation*}
        To simplify the notation, for each $i \in [m]$, we define the random vector $v_i \in \mathbb{C}^r$ as
\begin{equation*}
v_i := P_{\mathcal{U}} F^* D s_i.
\end{equation*}
Since $D$ is self-adjoint, the outer product is exactly $v_i v_i^* = P_{\mathcal{U}} F^* D s_i s_i^* D F P_{\mathcal{U}}^*$. This allows us to express the operator sum as
\begin{equation*}
P_{\mathcal{U}} F^* D S^* S D F P_{\mathcal{U}}^* = \sum_{i=1}^m v_i v_i^*.
\end{equation*}
        
      This allows us to write the expectation of each term in the sum as:
      \begin{equation}\notag
      \mathbb{E}[v_i v_i^* - P_{\mathcal{U}} F^*\hat{I} F P_{\mathcal{U}}^*] = P_{\mathcal{U}} F^* \mathbb{E}[D s_i s_i^* D - \hat{I}] F P_{\mathcal{U}}^* = 0.
      \end{equation}
      
      Consequently, the random matrices comprising the sum are mean-zero. Next, we bound their operator norms almost surely. Let us evaluate the Euclidean norm of $v_i$. Suppose the random vector $s_i$ realizes as the canonical basis vector $e_j$ for some $j \in \mathbb{N}$. We then have:
      \begin{align*}\|v_i \| &= \| P_\mathcal{U} F^* D e_j \| \\ 
      &= \frac{1}{\sqrt{p_j}} \| P_\mathcal{U} F^* e_j \| \\
      &= \frac{1}{\sqrt{p_j}} \sup_{x\in \mathcal{U} \cap B} |\langle x, F^* e_j \rangle| \\
      &\leq \mu_{\mathcal{U}}(F,p) \eqqcolon \mu,
      \end{align*}
      where the inequality follows directly from the definition of the local coherence $\alpha_j$ and the sampling parameter $\mu_{\mathcal{U}}(F,p)$. Since this bound holds for any realization $e_j$, we have $\|v_i\| \le \mu$ almost surely. We can now bound the operator norm of each mean-zero matrix:
      \begin{equation*}\| v_i v_i^* - P_{\mathcal{U}} F^* \hat{I} F P_{\mathcal{U}}^*\| \leq \|v_i v_i^*\| + \| P_{\mathcal{U}} F^* \hat{I} F P_{\mathcal{U}}^* \| \leq \mu^2 + 1 \leq 2\mu^2,
      \end{equation*}
      where we used the rank-one identity $\|v_i v_i^*\| = \|v_i\|^2 \leq \mu^2$, the operator norm bound $\| P_{\mathcal{U}} F^*\hat{I} F P_{\mathcal{U}}^* \| \leq 1$, and our assumption that $\mu \ge 1$.

      Then we compute
\begin{align*}
    \sigma^2 & = \left\|\sum_{i=1}^m \mathbb{E}\left[\left({v}_i {v}_i^* - P_{\mathcal{U}} F^*\hat{I} F P_{\mathcal{U}}^*\right)^2\right]\right\| \\
    & = \sup _{{u} \in \mathbb{C}^r \cap \mathbb{S}^{r-1}} \left\langle{u}, \sum_{i=1}^m\left(\mathbb{E}\left[{v}_i {v}_i^* {v}_i {v}_i^*\right] - \left(P_{\mathcal{U}} F^*\hat{I} F P_{\mathcal{U}}^*\right)^2 \right) {u}\right\rangle \\
    & \leq \sup _{{u} \in \mathbb{C}^r \cap \mathbb{S}^{r-1}} \left\langle{u}, \sum_{i=1}^m\left\|{v}_i\right\|^2 \mathbb{E}\left[{v}_i {v}_i^*\right] {u}\right\rangle - \sum_{i=1}^m \left\| P_{\mathcal{U}} F^*\hat{I} F P_{\mathcal{U}}^* {u}\right\|^2 \\
    & \leq \sup _{{u} \in \mathbb{C}^r \cap \mathbb{S}^{r-1}} \sum_{i=1}^m \mu^2 \left\langle{u}, \mathbb{E}\left[{v}_i {v}_i^*\right] {u}\right\rangle \\
    & \leq \mu^2 m.
\end{align*}
The second equality holds because the operator inside the norm is positive semi-definite, meaning its operator norm coincides with the supremum of its quadratic form over the unit sphere. The subsequent inequalities follow by dropping the non-positive subtracted term and applying the bound $\|\mathbb{E}[v_i v_i^*]\| = \|P_{\mathcal{U}} F^*\hat{I} F P_{\mathcal{U}}^*\| \leq 1$.

From this, we apply the Matrix Bernstein inequality (Lemma~\ref{bern}) with dimension $r$, variance parameter $\sigma^2 \leq m\mu^2$, and spectral bound $L \leq 2\mu^2$, obtaining:
\begin{equation*}
    \mathbb{P}\left\{\left\|\sum_{i=1}^m\left[P_{\mathcal{U}} F^* \left[ D s_i s_i^* D - \hat{I} \right]F P_{\mathcal{U}}^*\right]\right\| \geq \gamma \right\} \leq 2 r \exp \left(-\frac{\gamma^2 / 2} {m\mu^2 + \frac{2}{3} \mu^2 \gamma}\right).
\end{equation*}
This bound translates to the operator norm over the subspace. Since $P_{\mathcal{U}}^*$ is an isometry from $\mathbb{C}^r$ to $\mathcal{U}$, the condition above is equivalent to:
\begin{equation*}
    \mathbb{P}\left\{\sup _{{x} \in \mathcal{U}, \|x\|=1}\left|\frac{1}{m}\|SDF{x}\|^2- \| \hat{I} F {x}\|^2 \right| \geq \frac{\gamma}{m} \right\} \leq 2 r \exp \left(-\frac{\gamma^2 / 2} {m\mu^2 + \frac{2}{3} \mu^2 \gamma}\right).
\end{equation*}
Substituting $\delta = \frac{\gamma}{m}$, the exponent becomes:
\begin{equation*}
    -\frac{(m\delta)^2 / 2} {m\mu^2 + \frac{2}{3} \mu^2 m\delta} = -\frac{m\delta^2 } {2\mu^2 (1+ \frac{2}{3} \delta)}.
\end{equation*}
Assuming $\delta \leq 1$, we have $(1+\frac{2}{3}\delta) \leq 5/3 < 2$, which implies $\frac{1}{2(1+\frac{2}{3}\delta)} > \frac{1}{4}$. This yields the simplified bound:
\begin{equation*}
    \mathbb{P}\left\{\sup _{{x} \in \mathcal{U}, \|x\|=1}\left|\frac{1}{m}\|SDF{x}\|^2- \| \hat{I} F {x}\|^2 \right| \geq \delta \right\} \leq 2 r \exp \left(-\frac{m\delta^2 } {4\mu^2 }\right).
\end{equation*}
Finally, requiring the right-hand side to be at most $\varepsilon$,
\begin{equation*}
    2 r \exp \left(-\frac{m\delta^2 } {4\mu^2 }\right) \leq \varepsilon,
\end{equation*}
we conclude that the condition holds provided that
\begin{equation*}
    m \geq \frac{4\mu_{\mathcal{U}}(F,p)^2}{\delta^2}\log\left( \frac{2r}{\varepsilon}\right). \qedhere
\end{equation*}
    \end{proof}
\end{lemma}
We now present our main Generalized Restricted Isometry Property (Gen-RIP) result, which establishes that the restricted isometry holds for signals within the range of a $(k,d)$-generalized generative network.
\begin{theorem}[Gen-RIP]\label{gen-rip}
    Let $F$ be a unitary operator on $\ell^2(\mathbb{N})$ and let $p$ be a probability distribution on $\mathbb{N}$ with associated operators $S$, $D$, and $\hat{I}$. Let $G$ be a $(k, d)$-generalized generative network and $\mathcal{T}= R(G) - R(G)$. For $\varepsilon>0$, $\delta \in (0,1)$, assume the number of measurements satisfies
  $$m \ge \frac{4\mu_{\mathcal{T}}(F,p)^2}{\delta^2} \left[ 2kd \log\left(\frac{2e k_d}{k}\right) + \log\left(\frac{4k}{\epsilon}\right) \right],$$
    where we assume $\mu_{\mathcal{T}}(F,p)\geq 1$.
    Then, with probability at least $1-\varepsilon$, the following holds:
    \begin{equation*}
        \sup_{x\in \mathcal{T}, \|x\| =1 } \left| \frac 1 m\|S D F \boldsymbol{x}\|^2-\| \hat{I} F \boldsymbol{x}\|^2\right| \le \delta
    \end{equation*}
\begin{proof}
    By Lemma~\ref{remark-T}, $\mathcal T$ is a collection of at most $N^2$ $2k$-dimensional sub-spaces. Therefore we can write $\mathcal{T} = \bigcup_{j=1}^N \mathcal{L}_j$, and apply on each $\mathcal{L}_j$ Lemma~\ref{lemma-rip-gen} with $\varepsilon'>0$ and obtain
    \[
    \begin{aligned}
        & \mathbb{P}\left\{\sup _{x \in \mathcal{T}, \|x\|=1}\left|\frac 1 m\|S D F \boldsymbol{x}\|^2-\| \hat{I} F \boldsymbol{x}\|^2\right| \geq \delta\right\} \\
        & \quad \leq \sum_{j=1}^N \mathbb{P}\left\{\sup _{x \in \mathcal{L}_j ,\|x\|=1}\left| \frac 1 m\|S D F \boldsymbol{x}\|^2-\| \hat{I} F \boldsymbol{x}\|^2\right|\geq \delta\right\} \\
        & \quad \leq N \varepsilon' .
    \end{aligned}
    \]
    To find the requested rate for $m$,  we set $\epsilon' = \epsilon / N^2$ in Lemma~\ref{rate}, which requires:$$m \ge \frac{4\mu_{\mathcal{U}}(F,p)^2}{\delta^2} \log\left(\frac{2k N}{\epsilon}\right) = \frac{4\mu_{\mathcal{U}}(F,p)^2}{\delta^2} \left[ \log N^2 + \log\left(\frac{4k}{\epsilon}\right) \right].$$
    
    Substituting the bound $\log N \le kd \log\left(\frac{2e k_d}{k}\right)$ from Lemma~\ref{lemma-sottospazi} yields the desired sample complexity.    
\end{proof}
\end{theorem}
\subsection{Recovery bounds}
Now we have everything we need to prove the main theorem and corollary.
\begin{proof}[Proof of Theorem \ref{thm:main-recovery-merged}]
    Recall $x^{\perp} := x_0 - \Pi_{R(G)}(x_0)$. By the definition of the estimator and the triangle inequality, we have
    $$
    \begin{aligned}
        \|SDF \hat{x} - b\| & \leq \min _{x \in R(G)}\|SDF x - b\| + \hat{\varepsilon} \\
        & \leq \left\|SDF \Pi_{R(G)}(x_0) - b\right\| + \hat{\varepsilon} \\
        & = \left\|SDF x^{\perp} + \eta\right\| + \hat{\varepsilon} \\
        & \leq \left\|SDF x^{\perp}\right\| + \|\eta\| + \hat{\varepsilon}.
    \end{aligned}
    $$
    By assumption, the number of measurements $m$ is sufficiently large such that, by Corollary \ref{gen-rip}, the measurement operator $SDF$ satisfies the restricted isometry property over the difference set $\mathcal{T} = \Delta(R(G))$ with constant $\delta = 1/2$, with probability at least $1-\varepsilon$. Therefore, since the difference vector $\hat{x} - \Pi_{R(G)}(x_0) \in \Delta(R(G))$, we obtain
    $$
    \begin{aligned}
        \|SDF \hat{x} - b\| & = \left\|SDF\left(\hat{x} - \Pi_{R(G)}(x_0)\right) - SDF x^{\perp} - \eta\right\| \\
        & \geq \left\|SDF\left(\hat{x} - \Pi_{R(G)}(x_0)\right)\right\| - \left\|SDF x^{\perp}\right\| - \|\eta\| \\
        & \geq \sqrt{\frac{m}{2}}\left\|\hat{I}F \left(\hat{x} - \Pi_{R(G)}(x_0)\right)\right\| - \left\|SDF x^{\perp}\right\| - \|\eta\|.
    \end{aligned}
    $$
    Assembling the two inequalities gives
    $$
    \left\|\hat{I}F\left(\hat{x} - \Pi_{R(G)}(x_0)\right)\right\| \leq \sqrt{\frac{2}{m}}\left[2\left\|SDF x^{\perp}\right\| + 2\|\eta\| + \hat{\varepsilon}\right].
    $$
    Applying the triangle inequality and utilizing the unitarity of $F$, we derive the general bound:
    $$
    \begin{aligned}
        \left\|\hat{x} - x_0\right\| & \leq \left\|x_0 - \Pi_{R(G)}(x_0)\right\| + \left\|\hat{x} - \Pi_{R(G)}(x_0)\right\| \\
        & \leq \left\|x^{\perp}\right\| + \left\|\hat{I}F(\hat{x} - \Pi_{R(G)}(x_0))\right\| + \left\|\hat{I}^\perp F(\hat{x} - \Pi_{R(G)}(x_0))\right\| \\
        & \leq \left\|x^{\perp}\right\| + \frac{2\sqrt{2}}{\sqrt{m}}\left\|SDF x^{\perp}\right\| + \frac{2\sqrt{2}}{\sqrt{m}}\|\eta\| + \frac{\sqrt{2}}{\sqrt{m}}\hat{\varepsilon} + \left\|\hat{I}^\perp F(\hat{x} - \Pi_{R(G)}(x_0))\right\|.
    \end{aligned}
    $$
    This proves \eqref{equazione-madre}.
    
    For the second part, assume $p$ is admissible. We observe that $\hat{I}^\perp F u = 0$ for all $u \in \Delta(R(G))$. This holds because if the local coherence satisfies $\alpha_j > 0$, admissibility dictates that $p_j > 0$, implying the orthogonal projection component $\hat{I}^\perp_{j,j} = 0$. Conversely, if $\alpha_j = 0$, the $j$-th component is strictly zero by the definition of local coherence. Since $\hat{x} - \Pi_{R(G)}(x_0) \in \Delta(R(G))$, the tail term $\left\|\hat{I}^\perp F(\hat{x} - \Pi_{R(G)}(x_0))\right\|$ vanishes entirely, yielding \eqref{bound-admissible-simp}.
\end{proof}

\begin{proof}[Proof of Corollary \ref{cor:balancing-bound}]
    We simply need to observe that
    \begin{equation*}
        \left\|\hat{I}^\perp F(\hat{x}-\Pi_{R(G)}\left(x_0\right))\right\| = \left\|\hat{I}^\perp F\Pi_{R(G)}(\hat{x}-x_0)\right\| \leq \theta \| \hat{x} - x_0 \|,
    \end{equation*}
    simply substituting in \eqref{equazione-madre} we obtain the thesis.
\end{proof}
\section{Conclusions}\label{sec:conclusions} 
In this work we have established a theoretical foundation for infinite dimensional generative compressed sensing. By properly formulating the problem in $\ell^2(\mathbb{N})$ and accounting for the discretization gap we have shown that generative priors can be effectively used for functional recovery.
We introduced the concept of infinite dimensional local coherence and utilized it to derive optimal sampling strategies.
Our numerical experiments on Darcy flow validated these theoretical insights demonstrating that coherence based sampling significantly reduces the number of measurements required for accurate reconstruction.

Furthermore our empirical results revealed a valuable interaction between model resolution and sampling density. We found that in data scarce regimes employing a lower resolution generator with explicit upscaling provides a form of implicit regularization that stabilizes the inverse problem. Our findings suggest that the optimal generative prior is not necessarily the one with the highest fidelity, but rather the one that best matches the spectral content of the available measurements. By limiting resolution, we leverage the spectral bias of the network \cite{rahaman2019spectral} to stabilize the inversion in data-scarce regimes.

While our theoretical framework is robust, it relies on the assumption that the generator employs piecewise linear activations (e.g., ReLU) to ensure the range is contained within a union of finite-dimensional polyhedral cones. Extending this theory to networks with smooth activations, such as GeLU or SiLU, remains an open theoretical challenge. Additionally, while our theory provides error bounds in the presence of measurement noise, our numerical validations were primarily focused on the noiseless regime. Future work could explore the empirical robustness of this framework under varying noise levels and distributions. Finally, translating these infinite-dimensional Gen-RIP guarantees to other practical measurement operators (especially non-unitary), such as the Radon transform for continuous computed tomography, is the natural next step.

\section*{Acknowledgments}
This material is based upon work supported by the Air Force Office of Scientific Research under award numbers FA8655-23-1-7083. The research was supported in part by the MIUR Excellence Department Project awarded to Dipartimento di Matematica, Università di Genova, CUP D33C23001110001.  Co-funded by Regional Problem FSE+ (point 1.2 of attach. IX of Reg. (UE) 1060/2021). The research by MS has been supported by the MUR grants PRIN P2022XT498, funded by the European Commission under the NextGeneration EU programme. MS is a member of the ``Istituto Nazionale di Alta Matematica''.

\section*{Conflict of interest}
		The authors declare that they have no conflict of interest.

\bibliographystyle{plain}
\bibliography{cit}
\end{document}